\documentclass[12pt,reqno]{article}
\usepackage[utf8]{inputenc}
\usepackage{geometry}

\usepackage{tikz}

\usepackage{forest}

\usepackage{amsmath}
\usepackage{amssymb}
\usepackage{amsthm}
\usepackage{amsfonts}
\usepackage{mathtools}
\usepackage{commath} 
\usepackage{stackrel}
\usepackage{caption}

\usepackage{graphicx}
\usepackage{float}

\usepackage{xcolor}
\usepackage{verbatim}
\usepackage{enumitem}
\usepackage{hyperref}

\theoremstyle{plain}
\newtheorem{theorem}{Theorem}
\newtheorem*{theorem*}{Theorem}
\newtheorem{corollary}[theorem]{Corollary}

\newtheorem*{lemma*}{Lemma}

\theoremstyle{definition}
\newtheorem{definition}{Definition}

\theoremstyle{remark}

\DeclareMathOperator{\FS}{\mathsf{FS}}
\DeclareMathOperator{\DFS}{\mathsf{DFS}}
\DeclareMathOperator{\Tour}{\mathsf{Tour}}
\DeclareMathOperator{\Path}{\mathsf{Path}}
\DeclareMathOperator{\Cycle}{\mathsf{Cycle}}
\DeclareMathOperator{\des}{\mathsf{Des}}
\DeclareMathOperator{\cdes}{\mathsf{CDes}}
\DeclareMathOperator{\outdeg}{\mathsf{outdeg}}
\DeclareMathOperator{\ODP}{\mathsf{ODP}}
\newcommand{\conj}[1]{\overline{#1}}

\usepackage[
backend=biber,
style=numeric,
]{biblatex}
\title{Generalized Eulerian Numbers and Directed Friends-and-seats Graphs}

\author{David Dong}
\date{}

\begin{document}


\maketitle
\begin{abstract}
Let $A(n,m)$ denote the Eulerian numbers, which count the number of permutations on $[n]$ with exactly $m$ descents, or, due to the Foata transform, the number of permutations on $[n]$ with exactly $m$ excedances. Friends-and-seats graphs, also known as friends-and-strangers graphs, are a seemingly unrelated recent construction in graph theory. In this paper, we introduce directed friends-and-seats graphs and establish a connection between these graphs and a generalization of the Eulerian numbers. We use this connection to reprove and extend a Worpitzky-like identity on generalized Eulerian numbers.
\end{abstract}
\section{Introduction}\label{sec:intro}
The Eulerian numbers $A(n,m)$ are equal to the number of permutations of the numbers from $1$ to $n$ with exactly $m$ descents. It is well-known that $A(n,m)$ is also equal to the number of permutations of the numbers from $1$ to $n$ with exactly $m$ ascents, through a Foata transform \cite{Stanley}. 

Eulerian numbers are found in many different contexts. For instance, they and their generalizations come up in number theory, various combinatorial constructions, and simplicial complexes \cite{Petersen}. The Eulerian polynomials $A_n(x)$, which have the Eulerian numbers $A(n,m)$ as coefficients, also are involved in many identities. One of the most famous of these is Worpitzky's identity, which relates Eulerian polynomials to the sum of powers \cite{Graham}:
\[
\frac{A_{{n}}(x)}{(1-x)^{n+1}} = \sum_{m = 0}^{\infty} m^n x^m,
\]
which has multiple interesting combinatorial interpretations \cite{Spivey}.

Many different generalizations of Eulerian numbers and their properties have also been studied \cite{Me, Conger}. One particular generalization considers a graph $G$ with $n$ vertices labeled with the integers from $1$ to $n$ \cite{Barred}. Then, $G$-descents in a permutation are defined to be descents such that the two consecutive terms of the descent are also neighbors in the graph. If $G$ is a chordal graph, then a Worpitzky-like identity exists for the generalized Eulerian polynomial created by $G$.

We now move our attention to the \textit{friends-and-seats graph}, also known as \textit{friends-and-strangers graphs} in past papers \cite{Defant}. The name has been changed in this paper to be more fitting of their definition. These graphs are generated from any two graphs $X$ and $Y$ that both have $n$ vertices. We may treat each of the vertices of $Y$ as a person, such that two vertices in $Y$ are adjacent if and only if the corresponding two people are friends. Similarly, we can treat each of the vertices of $X$ as a seat, with two seats next to each other if and only if the corresponding vertices are adjacent.

The \textit{friends-and-seats graph} resulting from these two graphs, denoted $\FS(X,Y)$, is a graph with $n!$ vertices, where each vertex is labeled with a different bijection from the vertices of $X$ to the vertices of $Y$. Each vertex of a friends-and-seats graph can be thought of as an arrangement of friends sitting on seats. Say that two people are only allowed to swap seats if they are friends and the seats that they are next to each other. For every valid swapping of two friends, an edge is drawn between the vertices in the friends-and-seats graph corresponding to the arrangements of the friends and seats, before and after swapping.

The following is an example friends graph, labeled $X$, and seats graph, labeled $Y$.
\begin{figure}[h]
\vspace{-0.4cm}
\begin{center}
\includegraphics[width = 1.9in]{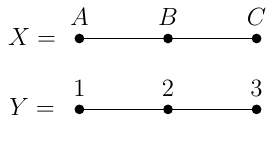}
\end{center}
\vspace{-0.5cm}
\captionsetup{justification=centering}
\caption{Example friends graph $(X)$ and seats graph $(Y)$, with people $A,B$ and $C$ and seats numbered from $1$ to $3$.}
\label{fig:fsinitial}
\end{figure}

\begin{figure}[ht!]
\vspace{-0.3cm}
\begin{center}
\includegraphics[width = 2.6in]{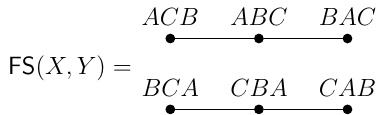}
\end{center}
\vspace{-0.3cm}
\captionsetup{justification=centering}
\caption{The resulting friends-and-seats graph. $ABC$ denotes that $A$ sits on chair 1, $B$ sits on chair 2, and $C$ sits on chair 3.}
\label{fig:fsfinal}

\end{figure}

Friends-and-seats graphs are relevant in many different contexts. As just one example, the friends-and-seats graph resulting from two complete graphs is isomorphic to the Cayley graph of the symmetric group generated by every transposition \cite{Defant}. They also connect with concepts from game theory. Each position of the famous 15-puzzle can be represented as a vertex of the friends-and-seats graph resulting from a $4$ by $4$ grid graph and a star graph. The fact that the resulting friends-and-seats graph is not connected proves that the 15-puzzle is not always solvable \cite{Defant}. As such, one well-studied property of friends-and-seats graphs is whether or not they are connected \cite{PRIMES22}.

Our paper is structured as follows. In Section~\ref{sec:Prelim}, we generalize friends-and-seats graphs into their directed version, and define the \textit{outdegree polynomial}, and motivate such a graph polynomial on directed friends-and-seats graphs. In Section~\ref{sec:genProp}, we prove general properties related to directed friends-and-seats graphs, and continue proving general properties related to the outdegree polynomial of these paths in \ref{sec:ODP Intro}. Finally, in Section~\ref{sec:ODP Examples}, we relate directed friends-and-seats graphs to Eulerian numbers, using them to prove a past theorem on generalized Eulerian numbers, and further extend them to prove a similar theorem on cyclic Eulerian numbers. 

\section{Preliminaries}\label{sec:Prelim}
In this section, we will introduce generalized Eulerian numbers and directed friends-and-seats graphs, and give motivation for defining the outdegree polynomial that connects the two. Then, we will list some general conventions and definitions that are relevant to our work, especially with regards to directed friends-and-seats graphs.

\subsection{Eulerian Numbers}

Let $S_n$ be the set of all bijections from $\{1,2,\ldots,n\}$ to $\{1,2,\ldots,n\}$. For any bijection $\sigma \in S_n$, define a \textit{descent} of $\sigma$ to be any integer $1 \leq i \leq n-1$ satisfying $\sigma(i) > \sigma(i+1)$. Furthermore, define an \textit{excedance} to be any integer $i \leq n$ satisfying $\sigma(i) > i$.

Formally define the Eulerian numbers $A(n,m)$ to be the number of bijections in $S_n$ with exactly $m$ descents. The corresponding Eulerian polynomials are defined to be
\[
A_{{n}}(x)=\sum _{{m=0}}^{{n}} A(n,m)\ x^{{m}}.
\]

We now formally define a generalization of the Eulerian numbers that was mentioned in the introduction \cite{Barred}. Consider any graph $G$ with $n$ vertices. For any bijection $\sigma$ sending $\{1,2,\ldots,n\}$ to $\{1,2,\ldots,n\}$, define the number of $G$-descents of $\sigma$, denoted $\des_G(\sigma)$, to be the number of positive integers $1 \leq i \leq n$ that satisfy the following conditions:
\begin{itemize}
\item The permutation $\sigma$ satisfies that $\sigma(i) > \sigma(i+1)$.
\item The vertices with labels $\sigma(i)$ and $\sigma(i+1)$ are neighbors in $G$.
\end{itemize} 
Recall that $S_n$ is the set of all bijections from $\{1,2,\ldots,n\}$ to $\{1,2,\ldots,n\}$. The generalized Eulerian polynomial $A_G(x)$ satisfies
\[
A_G(x) = \sum_{\sigma \in S_n} x^{\des_G(\sigma)}.
\]
Notably, if $G$ is a complete graph on $n$ vertices, $A_G(x)$ would be equal to the Eulerian polynomial $A_n(x)$. Define a graph to be chordal if every induced cycle has exactly three vertices. If $G$ is a chordal graph, then
\[
\frac{A_G(x)}{(1-x)^{n+1}} = (-1)^n \sum_{m=0}^{\infty} \chi_{\conj{G}}(-m-1) x^{m}, 
\]
where $\chi_{\conj{G}}(-m-1)$ is the chromatic polynomial of $G$, evaluated at $-m-1$. Notably, if $G$ is a complete graph, this statement is equivalent to Worpitzky's identity.

\subsection{Directed Friends-and-seats Graphs}
Recall the definition given for friends-and-seats graphs given in the introduction. To formalize this definition, we consider two graphs $X$ and $Y$ both with $n$ vertices. The friends-and-seats graph $\FS(X,Y)$ has $n!$ vertices, labeled with bijections from the vertices of $X$ to the vertices of $Y$. For any two bijections $\sigma$ and $\phi$, the vertices with labels $\sigma$ and $\phi$ in $\FS(X,Y)$ are connected by an edge if and only if there exist two vertices $a,b$ in $X$ that satisfy all of the following conditions:
\begin{itemize}
\item There is an edge between the vertices $a$ and $b$ in the graph $X$.
\item There is an edge between the vertices $\sigma(a)$ and $\sigma(b)$ in the graph $Y$.
\item The permutations satisfy that $\sigma (a) = \phi(b)$ and $\sigma (b) = \phi(a)$.
\item For all vertices $c$ in $X$ not equal to $a$ or $b$, we have $\sigma (c) = \sigma'(c)$.
\end{itemize}
The \textit{directed friends-and-seats graph} $\DFS(X,Y)$ directs every edge of a friends-and-seats graph. Let $X$ and $Y$ be arbitrary directed graphs with $n$ vertices. The graph $\DFS(X,Y)$ also has $n!$ vertices, each labeled with a different bijection between the vertices of $X$ and the vertices of $Y$. There is a directed edge between two vertices labeled with bijections $\sigma$ and $\phi$ in $\DFS(X,Y)$ if and only if there exist vertices $a$ and $b$ in $X$ that satisfy all of the following conditions:
\begin{itemize}
\item There is an edge directed from the vertex $a$ to the vertex $b$ in the graph $X$.
\item There is an edge directed from the vertex $\sigma(a)$ to the vertex $\sigma(b)$ in the graph $Y$.
\item The permutations satisfy that $\sigma (a) = \phi(b)$ and $\sigma (b) = \phi(a)$.
\item For all vertices $c$ in $X$ not equal to $a$ or $b$, we have $\sigma (c) = \sigma'(c)$.
\end{itemize}

Below is an example of a possible directed friends-and-seats graph resulting from given directed graphs $X$ and $Y$ both with three vertices.
\begin{figure}[H]
\vspace{-0.25cm}
\begin{center}
\includegraphics[width=2.6in]{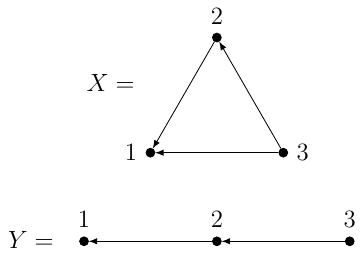}
\end{center}
\vspace{-0.5cm}
\captionsetup{justification=centering}
\caption{An example of a directed seat graph $X$ and friends graph $Y$.}
\label{fig:dfsginitial}
\end{figure}

\begin{figure}[H]
\vspace{-0.25cm}
\begin{center}
\includegraphics[width=3.9in]{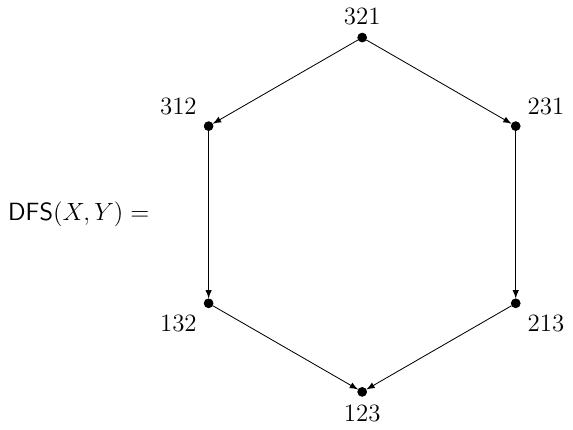}
\end{center}
\captionsetup{justification=centering}
\caption{The directed friends-and-seats graph $\DFS(X,Y)$ resulting from the seat graphs.}
\label{fig:dfsgfinal}
\end{figure}

Now, we will motivate directed friends-and-seats graphs by considering a particular example relating them to Eulerian numbers. 

Let $X$ be a transitive tournament graph with $n$ vertices labeled with the integers between $1$ and $n$, where edges are directed from vertices with greater labels to vertices with lower labels. We denote this graph as $\Tour_n$, which has edge set $\{(i \rightarrow j) \mid n \geq i > j \geq 1\}$. Let the vertices of the graph $Y$ also be labeled with the integers from $1$ to $n$, with the edge set $\{(i \rightarrow i+1) \mid 1 \leq i \leq n-1 \}$. $Y$ is also known as a directed path graph, denoted $\Path_n$. Since $\Tour_n$ and $\Path_n$ have vertices labeled with positive integers from $1$ to $n$, the graph $\DFS(\Tour_n,\Path_n)$ will have vertices labeled with permutations of the integers between $1$ and $n$.

Consider an arbitrary vertex $v \in \DFS(\Tour_n,\Path_n)$, labeled with the permutation $\pi_v$. By the definition of a friends-and-seats graph, the labels of the only possible vertices that could neighbor $v$ are of the form $\pi_v \circ (i, \; i+1)$ for $1 \leq i \leq n -1$. For each value of $i$, we know that $v$ is directed towards the vertex with label $\pi_v \circ (i \; i+1)$ if and only if in the graph $\Tour_n$, the vertex with label $\pi_v(i)$ is directed towards the vertex with label $\pi_v(i+1)$ in $X$. However, since $\Tour_n$ is a transitive tournament graph, this is the case if and only if $\pi_v(i) > \pi_v(i+1)$. 

This implies that the outdegree of the vertex $v$ in the graph $\DFS(\Tour_n, \Path_n)$ is equal to the number of descents of the vertex label $\pi_v$! Thus, we have found a relation between the outdegree of vertices in directed friends-and-seats graphs and the descents of a permutation. 

This inspires us to define the outdegree polynomial, denoted $\ODP(X,Y)$ for any arbitrary graphs $X$ and $Y$ both with $n$ vertices and calculated with the formula
\[
\ODP(X,Y) = \sum\limits_{v \in V(\DFS(X,Y))} x^{\outdeg(v)},
\]
where $\outdeg(v)$ denotes the outdegree of the vertex $v$ in $\DFS(X,Y)$, and $V(\DFS(X,Y))$ is the vertex set of the graph $\DFS(X,Y)$. Since the outdegree of the vertex $v$ in the graph $\DFS(\Tour_n, \Path_n)$ is equal to the number of descents of the vertex label $\pi_v$, we may say that
\[
A_n(x) = \ODP(\Tour_n, \Path_n)
\]
where $A_n(x)$ is the $n$th Eulerian polynomial.
\subsection{General Definitions and Conventions}
First, for any graph $X$, let $V(X)$ denote its vertex set and $E(X)$ its edge set. 

We usually assume that the graphs $X$ and $Y$ in $\DFS(X,Y)$ are regular directed graphs. However, all theorems and proofs in the following sections hold if any given graph is a directed multigraph. Directed multigraphs allow multiple edges between any two vertices $u$ and $v$, including the possibility that both edges $u \rightarrow v$ and $v \rightarrow u$ are present. If it is important for a theorem that a given graph is a multigraph, we explicitly mention so.

Define a directed graph to be \textit{acyclic} if it has no cycle of directed edges. Furthermore, define a vertex to be a \textit{sink} of a directed graph if it has outdegree $0$, and define a vertex to be a \textit{source} of a directed graph if it has indegree $0$. 

We commonly refer to a graph with $n$ vertices as being labeled with the positive integers between $1$ and $n$, which can be written with shorthand as $[n]$. A \textit{labeled acyclic graph} to be a graph with $n$ vertices labeled with $[n]$, satisfying that for all edges point from a vertex with a larger label to a vertex with a smaller label.  Note that any acyclic directed graph can be made into a labeled acyclic graph by assigning each vertex with the integers from $1$ to $n$. If there are multiple valid assignments, one is chosen arbitrarily. 

Since we assume that most graphs are labeled on $[n]$, we use some shorthand notation. Vertices in directed graphs with $n$ vertices will sometimes be referred to by numbers from $1$ to $n$. For example, the path graph $\Path_4$ can be said to have edge set $\{1 \rightarrow 2, 2 \rightarrow 3, 3 \rightarrow 4\}$. Furthermore, for a graph $X$ labeled on $[n]$ and any set $S \subseteq [n]$, the graph $X - S$ is defined to be the induced subgraph of $X$ that contains all vertices with labels not in the set $S$. The resulting directed friends-and-seats graph would thus have vertices that are bijections sending $[n]$ to $[n]$, or equivalently permutations of $[n]$.
 
For a labeled acyclic graph, we define the complement of $X$, denoted as $\conj{X}$, as the graph with vertex set $[n]$ and edge set $E(\Tour_n) \setminus E(X)$. In other words, for every $1 \leq i < j \leq n$, if the vertex with label $j$ is not directed towards the vertex with label $i$ in $X$, then $j \rightarrow i$ is an edge in $\conj{X}$.

There are several graphs that we name. Two graphs have already been mentioned: the transitive tournament graph $\Tour_n$ has vertex set $[n]$ and edge set $\{(i \rightarrow j) \mid n \geq i > j \geq 1\}$, and the directed path graph $\Path_n$ has vertex set $[n]$ and edge set $\{(i \rightarrow i+1) \mid 1 \leq i \leq n-1 \}$. Finally, the directed cycle graph $\Cycle_n$ has vertex set $[n]$ and edge set $\{(i \rightarrow i+1) \mid 1 \leq i \leq n-1 \} \cup \{n \rightarrow 1\}$. Notably, the graph $\Cycle_2$ is a multigraph, having edge set $\{1 \rightarrow 2, 2 \rightarrow 1\}$.

\section{General Properties}\label{sec:genProp}
In this section, we prove various properties of these $\DFS$ graphs that do not involve the outdegree polynomial. Some of these properties are similar to the ones seen in \cite{Defant} on regular friends-and-seats graphs.
\begin{theorem}
For any two directed graphs on $n$ vertices $X$ and $Y$, the graph $\DFS(X,Y)$ is automorphic to the graph $\DFS(Y,X)$.
\end{theorem}
\begin{proof}
Take any vertex $v$ in $\DFS(X,Y)$ with label $\pi_v$. The automorphism is the bijection that sends the vertex with label $\pi_v$ to its inverse, $\pi^{-1}_v$. We will show that every edge in $\DFS(X,Y)$ has a corresponding edge in $\DFS(Y,X)$. Consider any edge $u \rightarrow v$ in $\DFS(X,Y)$, where $\pi_u$ and $\pi_v$ are the labels of $u$ and $v$ respectively. Define $u^{-1}$ and $v^{-1}$ to be the vertices in $\DFS(Y,X)$ with labels $\pi^{-1}_u$ and $\pi^{-1}_v$ respectively. We will show that $u^{-1} \rightarrow v^{-1}$ is an edge in $\DFS(Y,X)$. 

By definition, $\pi_v$ must be of the form $\pi_u \circ (a \; b)$, where $a \rightarrow b$ is an edge in $X$ and $\pi_u(a) \rightarrow \pi_u(b)$ is an edge in $Y$. Note $\pi^{-1}_u \circ (\pi_u(a) \; \pi_u(b)) = \pi^{-1}_v$. We have already noted that $(\pi_u(a) \rightarrow \pi_u(b))$ is an edge in $Y$ and $(\pi^{-1}(\pi_u(a)), \pi^{-1}(\pi_u(b))) = (a,b)$ is an edge in $X$, so $u^{-1} \rightarrow v^{-1}$ is an edge in $\DFS(Y,X)$, as desired. 

By reversing $X$ and $Y$,we can see that if $u \rightarrow v$ is an edge in $\DFS(Y,X)$, then $u^{-1} \rightarrow v^{-1}$ is an edge in $\DFS(X,Y)$, which establishes our bijection.
\end{proof}
\begin{corollary}\label{cor:equal}
For any two directed graphs on $n$ vertices $X$ and $Y$, the polynomial $\ODP(X,Y)$ is equal to the polynomial $\ODP(Y,X)$.
\end{corollary}
\begin{theorem}\label{thm:subgraph}
Let $X'$ and $Y'$ be directed graphs on $n$ vertices. If $X$ is a subgraph of $X'$ and $Y$ is a subgraph of $Y'$, then $\DFS(X,Y)$ is a subgraph of $\DFS(X',Y')$. 
\end{theorem}
\begin{proof}
Consider any edge $u \rightarrow v$ between vertices $u,v \in V(\DFS(X',Y'))$. Say that $u$ and $v$ have labels` $\pi_u$ and $\pi_v$ respectively. By definition, there must exist some vertices $a,b \in V(X)$ satisfying $\pi_v = \pi_u \circ (a \; b)$, such that $a \rightarrow b$ is an edge in $X$ and $\pi_v(a) \rightarrow \pi_v(b)$ is an edge in $Y$. However, $a \rightarrow b$ must be an edge in $X'$ and $\pi_v(a) \rightarrow \pi_v(b)$ must also be an edge in $Y'$, which implies $u \rightarrow v$ is an edge in $\DFS(X',Y')$ as well.
\end{proof}

\begin{corollary}
For any two labeled acyclic graphs $X$ and $Y$ both on $n$ vertices, $\DFS(X,Y)$ is a subgraph of $\DFS(\Tour_n, \Tour_n)$.
\end{corollary}

\begin{theorem}\label{thm:acyc}
If $X$ and $Y$ are directed acyclic graphs with $n$ vertices, then $\DFS(X,Y)$ is also acyclic.
\end{theorem}
\begin{proof}
Say that $X$ and $Y$ are labeled acyclic graphs on $[n]$. Note that $X$ and $Y$ must both be subgraphs of the transitive tournament graph $\Tour_n$, so by Theorem~\ref{thm:subgraph} we can finish by showing that $\DFS(\Tour_n, \Tour_n)$ is acyclic. Now, for every permutation $\sigma \in \DFS(\Tour_n, \Tour_n)$, we define 
\[
f(\sigma) = \sum_{i=1}^{n} i\sigma(i). 
\]
We claim that if any two permutations $\sigma,\phi \in \DFS(\Tour_n, \Tour_n)$ satisfy that $\sigma \rightarrow \phi$, then $f(\sigma) > f(\phi)$, which suffices. Consider any two permutations $\sigma,\phi$ on $[n]$ with $\sigma \rightarrow \phi$ being an edge in $\DFS(\Tour_n, \Tour_n)$. By definition, $\phi$ can be expressed as $\sigma \circ (a \; b)$ for some vertices $a,b \in V(X)$ with $a \rightarrow b \in E(X)$ and $\sigma(a) \rightarrow \sigma(b) \in E(Y)$. Since $X$ and $Y$ are labeled acyclic graphs on $[n]$, we must have $a > b$ and $\sigma(a) > \sigma(b)$. However,
\[
f(\phi) = \sum_{i=1}^{n} i \phi(i) =  \sum_{i=1}^{n} i \sigma(i) - (a-b)(\sigma(a) - \sigma(b)) < f(\sigma)
\]
as desired.
\end{proof}
 Theorem~\ref{thm:acyc} is not generally true for graphs that are not acyclic. For instance, for graphs $X,Y$ labeled on $[3]$  with edge sets $E(X) = \{1 \rightarrow 2, 2 \rightarrow 3, 3 \rightarrow 1\}$ and $E(Y) = \{1 \rightarrow 2, 2 \rightarrow 3, 1 \rightarrow 3\}$, then $\DFS(X,Y)$ contains the cycle
 \[
 123 \rightarrow 213 \rightarrow 312 \rightarrow 321 \rightarrow 123.
 \]
\begin{figure}[H]
\begin{center}
\includegraphics[width=2.8in]{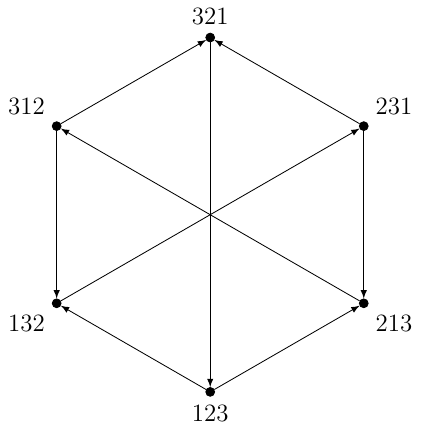}
\end{center}
\captionsetup{justification=centering}
\caption{The friends-and-seats graph $\FS(X,Y)$ resulting from the graphs described above.}
\label{fig:cyclexample}
\end{figure}

\section{Outdegree Polynomial}\label{sec:ODP Intro}
Recall that in Section~\ref{sec:Prelim} we defined the outdegree polynomial on the two graphs $X$ and $Y$ to be

\[
\ODP(X,Y) = \sum\limits_{v \in V(\DFS(X,Y))} x^{\outdeg(v)}.
\]
In this section, we examine various ways to determine the value of this polynomial for graphs $X$ and $Y$, mostly through the lens of adding and removing edges.

For any vertices $a,b \in X$, define $\ODP(X,Y)_{a \rightarrow b}$ to be the sum of $x^{\outdeg(v)}$ for all vertices $v \in V(\DFS(X,Y))$ with labels $\sigma$ such that $\sigma(a) \rightarrow \sigma(b)$ is an edge in $Y$. Notably, if $a \rightarrow b \in E(X)$, then  $x \mid \ODP(X,Y)_{a \rightarrow b}$.

Furthermore, define $\ODP_{\sigma(i) = j}$ to be the sum of $x^{\outdeg(v)}$ all vertices $v$ with labels $\sigma$ satisfying $\sigma(i) = j$. Finally, for any graph $X$ labeled on $[n]$ and arbitrary edge $u \rightarrow v \in E(X)$, define $X^{uv}$ to be the multigraph labeled on $[n] \setminus \{v\}$ such that any edge of the form $\{i \rightarrow v\}$ or $\{v \rightarrow j\}$ in $X$ has been replaced with $\{i \rightarrow u\}$ and $\{u \rightarrow j\}$ respectively in $X^{uv}$.

\begin{theorem}\label{thm:odp}
Consider any two directed graphs $X,Y$ with $V(X) = V(Y) = [n]$. For an arbitrary edge $\{a \rightarrow b\} \in E(X)$, define $X'$ to be the graph with vertex set $[n]$ and edge set $E(X) = E(X') \setminus \{a \rightarrow b\}$. We have:
\[
\ODP(X',Y) = \ODP(X,Y) - \left(\frac{x-1}{x}\right) \ODP(X,Y)_{a \rightarrow b}.
\]
\end{theorem}
\begin{proof}
Consider a vertex $v \in V(\DFS(X,Y))$ with corresponding permutation $\pi_v$. If $\pi_v(a) \rightarrow \pi_v(b)$ is not an edge in $Y$, then $\outdeg(v)$ is the same in both $\DFS(X,Y)$ and $\DFS(X',Y)$. Otherwise, the outdegree of $v$ decreases by one when going from $\DFS(X,Y)$ to $\DFS(X',Y)$. Thus, we have 
\begin{align*}
\ODP(X',Y) &= \ODP(X,Y) - \ODP(X,Y)_{a \rightarrow b} + \ODP(X',Y)_{a \rightarrow b} \\
&= \ODP(X,Y) - \ODP(X,Y)_{a \rightarrow b} + \frac{1}{x} \ODP(X,Y)_{a \rightarrow b} \\
&= \ODP(X,Y) - \left(\frac{x-1}{x}\right) \ODP(X,Y)_{a \rightarrow b}.
\end{align*}
as expected.
\end{proof}
Unfortunately, calculating $\ODP(X,Y)_{a \rightarrow b}$ is often very difficult. The next couple of theorems prove that it is doable under certain graphs $X$ and certain properties of $a$ and $b$. We define a subset $S$ of the vertex labels of $[n]$ is said to be \textit{self-equivalent} if for any fixed element $t \in [n] \setminus S$ either every edge in $S$ is directed towards $t$, $t$ is directed towards every edge in $S$, or there is no edge between any element of $S$ and $t$. 
\begin{theorem}
Let $\{a, b\}$ be a self-equivalent set of vertex labels in a directed graph $X$ labeled on $[n]$, where $a \rightarrow b \in E(X)$. Then, 
\[
\ODP(X,Y)_{a \rightarrow b} = x \ODP(X,Y)_{b \rightarrow a}
\]
\end{theorem}
\begin{proof}
Let $S$ be the set of all permutations $\sigma$ on $[n]$ satisfying that $\sigma(a) \rightarrow \sigma(b) \in Y$. Similarly, let $T$ be the set of all permutations $\sigma$ on $[n]$ satisfying that $\sigma(b) \rightarrow \sigma(a) \in Y$.  Note that for any $\sigma \in S$, there exists a corresponding permutation $\sigma \circ (a \; b) \in T$. 

Consider any permutation $\sigma \in S$ and the corresponding vertex $v \in \DFS(X,Y)$ with label $\sigma$. Define $\sigma' = \sigma \circ (a \; b)$ and let $v'$ be the vertex in $\DFS(X,Y)$ with label $\sigma'$. We prove that $\outdeg(v) = \outdeg(v') + 1$, which finishes by summing over all $\sigma \in S$. Say that $v \rightarrow v \circ (i \; j)$ is an edge in $\DFS(X,Y)$, for some vertices $i,j \in V(X)$. 
\begin{itemize}
\item If $i,j \not\in \{a,b\}$, it is clear that $v' \rightarrow v' \circ(i \; j)$ is also an edge in $\DFS(X,Y)$.
\item If exactly one of $i,j$ is $a$ or $b$, we can note by the definition of self-equivalent that $v' \rightarrow v' \circ(i \; j)$ must also be edge in $\DFS(X,Y)$.
\item Note that when $i = a$ and $j = b$, there is an edge from $v \rightarrow v'$ as $v \circ (a \; b) = v'$.   
\end{itemize}
Thus, all outward edges correspond except for $1$, implying that $v$ always has one higher outdegree than $v'$.
\end{proof}
We define two relaxed conditions related to self-equivalency. Consider a graph $X$ labeled on $[n]$. A subset $S \subseteq [n]$ is \textit{sink-equivalent} if for all $t \in [n] \setminus S$, either all $s \in S$ are directed towards $t$ or no $s \in S$ are directed towards $t$ in the graph $X$. Similarly, a subset $S \subseteq [n]$ is \textit{source-equivalent} if for every $t \in [n] \setminus S$, either $t$ is directed towards all $s \in S$ or $t$ is directed towards no $s \in S$ in the graph $X$. Notably, a set is self-equivalent if and only if it is both sink-equivalent and source-equivalent. 
\begin{theorem}\label{thm:pointsquishing}
Consider two directed graphs (possibly multigraphs) $X,Y$ labeled on $[n]$, and a sink-equivalent set $\{a,b\}$ of $X$ such that $\{a \rightarrow b\} \in E(X)$. Let $X'$ be the induced subgraph of $X$ with vertex set $[n] \setminus b$.  We have:
\[
\ODP(X,Y)_{a \rightarrow b} = x\sum_{\{u\rightarrow v\} \in E(Y)} \ODP(X',Y^{uv})_{\sigma(a) = u}.
\]
\end{theorem}
\begin{proof}
For every edge $u \rightarrow v \in E(Y)$, we will define $\ODP(X,Y)_{\sigma(a) = u, \sigma(b) = v}$ to be the sum of $x^{\outdeg(v)}$ for every vertex $v \in \DFS(X,Y)$ with label $\sigma$ satisfying $\sigma(a) = u$ and $\sigma(b) = v$. Summing over every edge in $Y$, we have:
\[
\ODP(X,Y)_{a \rightarrow b}  = \sum_{\{u\rightarrow v\} \in E(Y)}\ODP(X,Y)_{\sigma(a) = u, \sigma(b) = v}.
\]
Thus, it suffices to prove that 
\[
\ODP(X,Y)_{\sigma(a) = u, \sigma(b) = v} = x\ODP(X',Y^{uv})_{\sigma(a) = u}.
\]
To do this, we create a bijection between vertices $v \in \DFS(X,Y)$ with labels $\sigma$ satisfying the restrictions on the left-hand side and vertices $v' \in \DFS(X',Y)$ with labels $\phi$ satisfying the restriction on the right-hand side. In particular, we must have that the outdegree of $v$ in $\ODP(X,Y)$ is one more than the outdegree of $v'$ in $\ODP(X',Y^{uv})$. To do this, consider any valid permutation $\sigma$ on $[n]$ satisfying that $\sigma(a) = u$ and $\sigma(b) = v$, and let $\phi(i) = \sigma(i)$ for all $i \in [n] \setminus \{b\}$. Note that $\phi$ is a bijection from $[n] \setminus \{b\}$ to $[n] \setminus \{v\}$, as expected.

We claim that there is a bijection from edges going out of $v$ to edges going out of $v'$, with one exception. This would show that $v$ has one higher outdegree than $v'$, finishing. Consider any integers $1 \leq i < j \leq n$ such that there is an edge going from $v$ to the vertex with label $\sigma \circ (i \; j)$.
\begin{itemize}
\item If $i,j$ are both not $b$, then there exists a corresponding edge with vertex labels $\phi \rightarrow \phi \circ (i \; j)$.
\item If $i = b$ (and $j \neq a$), then we may replace $b$ with $a$ to get a corresponding edge with vertex labels $\phi \rightarrow \phi \circ (a \; j)$, because $a$ and $b$ are sink-equivalent and every edge from $\sigma(i) = b \rightarrow j$ has been replaced with one from $a \rightarrow j$ in $Y^{uv}$.
\item Similarly, if $j = b$ and $i \neq a$, then we again may replace $b$ with $a$.
\item Finally, there is an edge $\sigma \rightarrow \sigma \circ (a \; b)$ that does not correspond with one in $\phi$.
\end{itemize}
\end{proof}
\section{Generalized Eulerian Numbers}\label{sec:ODP Examples}
In this section, we use results from previous sections to prove statements about generalized Eulerian numbers. First, we define the notion of a \textit{perfect elimination ordering}. We will then use this definition to reprove a theorem relating generalized Eulerian numbers to the chromatic number of graphs. We will then further generalize this theorem.
\subsection{Perfect Elimination Ordering}
Perfect elimination orderings and directed chordal graphs are defined as the following:
\begin{definition}
Let $X$ be a labeled acyclic graph. Then, the labeling $(1,2,\ldots,n)$ of the vertices of the graph $X$ is a \textit{perfect elimination ordering} of $X$ if for any $1 \leq i < j \leq n$, if $j$ is directed to $i$, then there is a clique containing all vertices between $i$ and $j$ inclusive, i.e., for any $j \geq b > a \geq i$, $b$ is directed towards $a$. A labeled acyclic graph is a \textit{directed chordal} graph if there exists a labeling of its vertices that is a perfect elimination ordering.
\end{definition}

For instance, a transitive tournament graph is a directed chordal graph. We can construct the complement of all directed chordal graphs by removing sink-equivalent edges:
\begin{theorem}\label{thm:chordalSinkEquivalent}
Let $X$ be a labeled acyclic graph on $[n]$ such that its complement, denoted $\conj{X}$, is a directed chordal graph with $k$ edges. Then, there exists a sequence of graphs
\[
\Tour_n = X_0, X_1, X_2, \ldots, X_k = X
\]
such that for all $1 \leq i \leq k$ there exists a sink-equivalent set $\{a_i,b_i\}$ such that $E(X_{i-1}) = E(X_i) \cup \{a_i \rightarrow b_i\}$, and $\conj{X_i}$ is a directed chordal graph.
\end{theorem}
\begin{proof}
It suffices to show that a graph $X_{k-1}$ exists where $\conj{X_{k-1}}$ is chordal and $E(X_{k-1}) = E(X_k) \cup \{a_k \rightarrow b_k\}$ for some sink-equivalent set $\{a_k,b_k\}$. If this is possible, then we may repeat this process with $X_{k-1}, X_{k-2}, \ldots, X_1$ until we reach $\Tour_n$.

We pick $b_k \rightarrow a_k$ to be the edge in $\conj{X_k}$ such that $a_k$ is the smallest vertex with positive indegree, and $b_k$ is the largest label of a vertex that has a directed edge to $a_k$. Again, let $X_{k-1}$ have the same vertex set as $X_k$ with $E(X_{k-1}) = E(X_k) \cup \{a_k \rightarrow b_k\}$. We need to prove that $\conj{X_{k-1}}$ is a directed chordal graph and that $\{a_k,b_k\}$ is a sink-equivalent set.

We first prove that $\conj{X_{k-1}}$ is a directed chordal graph. Without loss of generality, assume that $[n]$ is a perfect elimination ordering of the labels of the vertices of $\conj{X_k}$. Assume for the sake of contradiction that $[n]$ is not also a perfect elimination ordering of the labels of the vertices of $\conj{X_{k-1}}$. Then, there must exist some vertices with labels $i$ and $j$ such that $j$ is directed towards $i$, and vertices with labels $a$ and $b$ with $j \geq b > a \geq i$ such that $b$ is not directed towards $a$. 

Since $[n]$ is a perfect elimination ordering of the labels of the vertices $\conj{X_k}$, it must be the case that $b = b_k$ and $a = a_k$, as this was the only removed edge. Thus, $i \leq a_k$ and $j \geq b_k$. However, note $a_k$ is the smallest vertex with positive outdegree, so it must be the case that $i = a_k$. Furthermore, $b_k$ is the largest label of a vertex that has a directed edge to $a_k$, so it must be the case that $j = b_k$, contradiction.

We now prove that $S_k = \{a_k,b_k\}$ is a sink-equivalent set of $X_k$. Assume for the sake of contradiction that there exists a vertex $i$ such that one of the elements of $S_k$ is directed to $i$, but not the other. Since $[n]$ is a perfect elimination ordering of the labels of the vertices $\conj{X_k}$, if $i \geq b_k$ neither $a_k$ nor $b_k$ can be directed to the vertex with label $i$ in the graph $X_k$. As such, $i < b_k$. 

However, if $a_k$ is not directed to the vertex with label $i$ in $X_k$, then $a_k$ is directed to the vertex with label $i$ in $\conj{X_k}$, implying that $a_k$ is not the smallest vertex with positive outdegree in $\conj{X_k}$, contradiction. The exact same argument can be made with $b_k$, showing that both $a_k$ and $b_k$ are directed toward the vertex with label $i$, contradiction.

\end{proof}
\subsection{Path Graph}
Recall that a path graph $\Path_n$ has $n$ vertices and edge set $\{i \rightarrow (i+1) \mid 1 \leq i \leq n-1 \}$. Furthermore, recall the generalization of the Eulerian numbers mentioned in Section~\ref{sec:Prelim} \cite{Barred}. The following theorem was mentioned:
\begin{theorem*}
Let $\chi_{G}(x)$ be the chromatic polynomial of the undirected graph $G$, and let $A_G(x)$ be the generalized Eulerian polynomial on $G$, as described in Section~\ref{sec:intro}. We have:
\[
\frac{A_G(x)}{(1-x)^{n+1}} = (-1)^n \sum_{m=0}^{\infty} \chi_{\conj{G}}(-m-1) x^{m} 
\]
where $\conj{G}$ is the conjugate of the graph $G$.
\end{theorem*}

This theorem can be restated as one involving directed friends-and-seats graphs, as shown below. We will prove this theorem with the help of Theorems~\ref{thm:pointsquishing} and \ref{thm:chordalSinkEquivalent}.
\begin{theorem}\label{thm:gessel}
Let $X$ be a directed acyclic graph with vertex set $[n]$ and edge set $E(X)$ such that $[n]$ is a perfect elimination ordering of $X$. Then, 
\[
\frac{\ODP(X, \Path_n)}{(1-x)^{n+1}} = (-1)^n \sum_{m = 0}^{\infty} \chi_{\conj{X}}(-m-1) x^m,
\]
where $\chi_{\conj{X}}(m)$ is the chromatic polynomial of the graph $\conj{X}$, ignoring the direction of each edge.

\end{theorem}

\begin{proof}
We first show that for a sink-equivalent set $\{a,b\}$ we have that $\ODP(X,\Path_n)_{a \rightarrow b} = x \ODP(X \setminus \{ b \}, \Path_{n-1})$.
By Theorem~\ref{thm:pointsquishing}, we have 
\[
\ODP(X,\Path_n)_{a \rightarrow b} = x \sum_{(u\rightarrow v) \in E(\Path_i)} \ODP(X \setminus \{ b \},\Path_{n}^{uv})_{\sigma(a) = u}.
\]
First, note that for $u\rightarrow v \in E(\Path_n)$ to be true, it must be the case that $v = u+1$. Thus, $\Path_{n}^{uv}$ has edge set 
\begin{align*}
&\{(1 \rightarrow 2), (2\rightarrow 3), \ldots, (u-1 \rightarrow u)\} \\
&\cup \{(u \rightarrow v+1)\} \\
&\cup \{(v+1 \rightarrow v+2), (v+2 \rightarrow v+3), \ldots, (n-1 \rightarrow n)\}
\end{align*}
which, after relabeling edges by subtracting one from the vertices with labels $v,v+1,\ldots n$, is exactly the edge set of $\Path_{n-1}$. We may therefore say that 
\[
\ODP(X \setminus \{ b \},\Path_{n}^{uv})_{\sigma(a) = u} = \ODP(X \setminus \{ b \},\Path_{n-1})_{\sigma(a) = u}.
\]
Now, we write:
\begin{align*}
\ODP(X,\Path_n)_{a \rightarrow b} &= x \sum_{(u\rightarrow v) \in E(Y)} \ODP(X \setminus \{ b \},\Path_{n-1})_{\sigma(a) = u} \\
&= x \sum_{u \in V(\Path_{n-1})} \ODP(X \setminus \{ b \},\Path_{n-1})_{\sigma(a) = u} \\
&= x \ODP(X \setminus \{ b \}, \Path_{n-1}),
\end{align*}
as desired. 

We now induct on $n$, the number of vertices of $X$. The base case, when $n = 1$, is trivial. Assume that our theorem is true for all $i \leq n$. Then, by Theorem~\ref{thm:chordalSinkEquivalent} pick a sequence of graphs $\Tour_n = X_0, X_1, \ldots, X_k = X$ such that for all $1 \leq i \leq k$ the $E(X_{i-1}) = E(X_i) \cup \{a_i \rightarrow b_i\}$ for some sink-equivalent set $\{a_i \rightarrow b_i\}$. For all $1 \leq i \leq k$, we note that
\[
\frac{\ODP(X_i,\Path_n)}{(1-x)^{n+1}} = \frac{\ODP(X_{i-1},\Path_n)}{(1-x)^{n+1}} + \frac{\ODP(X_{i-1},\Path_n)_{a \rightarrow b}}{x(1-x)^n} 
\]
by Theorem~\ref{thm:odp}. We now induct on $i$. The base case $i = 0$ is equivalent to
\[
\frac{\ODP(\Tour_n, \Path_n)}{(1-x)^{n+1}} = \sum_{m = 0}^{\infty} (m+1)^n x^m.
\]
Note that $\ODP(\Tour_n, \Path_n)$ is the Eulerian polynomial $A_n$, and so this statement becomes a well-known statement on Eulerian polynomials \cite{Graham}. Now, for all $1 \leq i \leq k$, we use both of our inductive hypotheses:
\begin{align*}
\frac{\ODP(X_i,\Path_n)}{(1-x)^{n+1}} &= \frac{\ODP(X_{i-1},\Path_n)}{(1-x)^{n+1}} + \frac{\ODP(X_{i-1},\Path_n)_{a \rightarrow b}}{x(1-x)^n} \\
&= (-1)^n \sum_{m = 0}^{\infty} \chi_{\conj{X_{i-1}}}(m) x^m + \frac{\ODP(X_{i-1} \setminus \{b\},\Path_{n-1})}{(1-x)^n} \\
&= (-1)^n \sum_{m = 0}^{\infty} \left(\chi_{\conj{X_{i-1}}}(-m-1) - \chi_{\conj{X_{i-1} \setminus \{b\}}}(-m-1)  \right) x^m. 
\end{align*}
However, note that for any positive integer $k$, the value $\chi_{\conj{X_{i-1} \setminus \{b\}}}(k)$ is exactly the number of ways to color the graph $\conj{X_{i-1}}$ with $k$ colors such that $a$ and $b$ share the same color but no two other neighboring vertices share the same color. This is also $\chi_{\conj{X_{i-1}}}(k) - \chi_{\conj{X_{i}}}(k)$. Thus:
\[
(-1)^n \sum_{m = 0}^{\infty} \left(\chi_{\conj{X_{i-1}}}(-m-1) - \chi_{\conj{X_{i-1} \setminus \{b\}}}(-m-1)  \right) x^m = (-1)^n \sum_{m = 0}^{\infty} \chi_{\conj{X_{i}}}(-m-1) x^m 
\]
which proves our inductive step.
\end{proof}

\subsection{Cycle Graph}
Cyclic Eulerian numbers $C(n,m)$ are count the number of permutations $\sigma$ on $[n]$ with $m$ descents, where a descent is also counted if $\sigma(n) > \sigma(1)$. Just like their Eulerian counterparts, their properties have been extensively studied \cite{Cellini}. We may generalize these numbers using $G$-cyclic-descents for any graph $G$, which are very similar to their $G$-descent counterparts. We study this generalization under the formulation of directed friends-and-seats graphs.

Recall that a directed cycle graph $\Cycle_n$ has $n$ vertices labeled on $[n]$ and edge set $\{i \rightarrow (i+1) \mid 1 \leq i \leq n-1 \} \cup \{n \rightarrow 1\}$. Our first theorem considers the equivalent statement of Worpitzky's identity for cyclic Eulerian numbers, using directed friends-and-seats graphs:
\begin{theorem}\label{thm:cycle-init}
For all positive integers $n$, we have that
\[
\frac{\ODP(\Tour_n, \Cycle_n)}{(1-x)^{n}} = n \sum_{m = 0}^{\infty} m^{n-1} x^m,
\]
where, by convention, $0^0 = 1$.
\end{theorem}
\begin{proof}
We may manually verify that $n = 1$ results in $\frac{1}{1-x} = \sum_{m = 0}^{\infty} x^m$, which is true. For $n > 1$, we will use that 
\[
\frac{\ODP(\Tour_n, \Path_n)}{(1-x)^{n+1}} = \sum_{m = 0}^{\infty} (m+1)^n x^m,
\]
which, as previously mentioned, is a well-known statement on Eulerian polynomials \cite{Graham}. As such, for all $n \geq 1$, the theorem statement is equivalent to 
\[
\ODP(\Tour_{n}, \Cycle_{n}) = nx\ODP(\Tour_{n-1}, \Path_{n-1}).
\]
Now, for every $1 \leq i \leq n-1$, we prove that 
\[
\ODP(\Tour_{n}, \Cycle_{n})_{\sigma(i) = n} = \ODP(\Tour_{n}, \Cycle_{n})_{\sigma(1) = n}.
\]
We construct a bijection between permutations $\sigma$ of $[n]$ satisfying $\sigma(i) = n$ with outdegree $k$, and permutations $\phi$ of $[n]$ satisfying $\phi(1) = n$ with outdegree $k$. Let $\phi$ be the rotation of $\sigma$, where 
\[
\phi(j) = \sigma(j + i - 1)
\]
for every $1 \leq j \leq n$, where indices are taken $\bmod \: n$ and $0 = n$. Then, any outward edge from $\sigma$ to $\sigma \circ (j \; k)$ has a corresponding edge from $\phi$ to $\phi \circ (j + i - 1 , k + i - 1)$ (where again indices are taken $\bmod \: n$ and $0 = n$.) Therefore,
\[
\ODP(\Tour_{n}, \Cycle_{n}) = \sum_{i=1}^{n} \ODP(\Tour_{n}, \Cycle_{n})_{\sigma(i) = n} = n \ODP(\Tour_{n}, \Cycle_{n})_{\sigma(1) = n}.
\]
We now construct a bijection between permutations $\sigma$ on $[n]$ satisfying $\sigma(1) = n$ with outdegree $k$ in $\ODP(\Tour_{n}, \Cycle_{n})$, and permutations $\phi$ on $[n-1]$ with outdegree $k-1$, which would show that 
\[
n \ODP(\Tour_{n}, \Cycle_{n})_{\sigma(1) = n} = nx\ODP(\Tour_{n-1}, \Path_{n-1}),
\]
and prove the theorem. To do this, take $\phi(i) = \sigma(i-1)$ for $2 \leq i \leq n$. Consider any edge $\sigma \rightarrow (i \; j)$. If $i = n$, then $j$ must be $\sigma^{-1}(2)$, and there is exactly one edge; clearly $j$ cannot be $n$. Otherwise, there is a corresponding directed edge from $\phi$ to $\phi \rightarrow (i \; j)$, and $\phi$ has one less outdegree than $\sigma$.
\end{proof}
Now, we may prove the cyclic equivalent of Theorem~\ref{thm:gessel}.
\begin{theorem}\label{thm:cycle-gessel}
Let $X$ be a directed acyclic graph with vertex set $[n]$ and edge set $E(X)$ such that $[n]$ is a perfect elimination ordering of $X$. Then, 
\[
\frac{\ODP(X, \Cycle_n)}{(1-x)^{n}} = (-1)^n n \sum_{m = 0}^{\infty} \frac{\chi_{\conj{X}}(-m)}{m} x^m.
\]
By symmetry, the chromatic polynomial $\chi_{\conj{X}}(x)$ is always divisible by $x$. Thus, when $m = 0$, the value of $\frac{\chi_{\conj{X}}(-m)}{m}$ is defined as the value of the polynomial $\frac{\chi_{\conj{X}}(-m)}{m}$ evaluated at $m = 0$.
\end{theorem}
\begin{proof}
This proof follows a very similar path to the one seen in Theorem~\ref{thm:gessel}. We show that for a sink-equivalent set $\{a,b\}$ and an integer $n \geq 2$ we have 
\[
\ODP(X,\Cycle_n)_{a \rightarrow b} = \left(\frac{nx}{n-1}\right)  \ODP(X \setminus \{ b \}, \Cycle_{n-1}).
\]
By Theorem~\ref{thm:pointsquishing}, we have 
\[
\ODP(X,\Cycle_n)_{a \rightarrow b} = x \sum_{(u\rightarrow v) \in E(\Cycle_i)} \ODP(X \setminus \{ b \},\Cycle_{n}^{uv})_{\sigma(a) = u}.
\]
First, note that for $u\rightarrow v \in E(\Cycle_n)$ to be true, it must be the case that $v = u+1$, or $u = n$ and $v = 1$. In the case where $v = u+1$, then $\Cycle_{n}^{uv}$ has edge set 
\begin{align*}
&\{(1 \rightarrow 2), (2\rightarrow 3), \ldots, (u-1 \rightarrow u)\} \\
&\cup \{(u \rightarrow v+1)\} \\
&\cup \{(v+1 \rightarrow v+2), (v+2 \rightarrow v+3), \ldots, (n-1 \rightarrow n), (n \rightarrow 1)\}.
\end{align*}
If we relabel edges by subtracting one from the labels $v,v+1,\ldots, n$, this is exactly the edge set of $\Cycle_{n-1}$. If $u = n$ and $v = 1$, then $\Cycle_{n}^{uv}$ has edge set 
\[
\{(i \rightarrow i+1) \mid 2 \leq i \leq n-1 \} \cup \{(n \rightarrow 2)\},
\]
which also is the edge set of $\Cycle_{n-1}$ after relabeling edges by subtracting one from the labels of all vertices. We may therefore say that 
\[
\ODP(X \setminus \{ b \},\Cycle_{n}^{uv})_{\sigma(a) = u} = \ODP(X \setminus \{ b \},\Cycle_{n-1})_{\sigma(a) = u}
\]
for $1 \leq u \leq n-1$, and 
\[
\ODP(X \setminus \{ b \},\Cycle_{n}^{uv})_{\sigma(a) = u} = \ODP(X \setminus \{ b \},\Cycle_{n-1})_{\sigma(a) = n-1}
\]
when $u = n$. We may relabel the vertices of a cycle by adding or subtracting any arbitrary integer from the label of each vertex, $\pmod{n}$ (where the vertex with label $0$ is relabeled to have label $n$). As such, we may say that
\[
\ODP(X \setminus \{ b \},\Cycle_{n-1})_{\sigma(a) = u} = \ODP(X \setminus \{ b \},\Cycle_{n-1})_{\sigma(a) = 1}
\]
for all $1 \leq u \leq n$. Now, we may simplify to find that
\begin{align*}
\ODP(X,\Cycle_n)_{a \rightarrow b} &= x \sum_{(u\rightarrow v) \in E(Y)} \ODP(X \setminus \{ b \},\Cycle_{n-1})_{\sigma(a) = u} \\
&= nx  \ODP(X \setminus \{ b \},\Cycle_{n-1})_{\sigma(a) = 1} \\
&= \frac{nx}{n-1} \left( \sum_{i = 1}^{n-1} \ODP(X \setminus \{ b \},\Cycle_{n-1})_{\sigma(a) = i} \right) \\
&= \left(\frac{nx}{n-1}\right)  \ODP(X \setminus \{ b \}, \Cycle_{n-1}),
\end{align*}
as desired. 

Our inductive step is also very similar to the inductive step found in the proof of Theorem~\ref{thm:gessel}. We induct on $n$, the number of vertices of $X$. $n = 1$ is trivial. Assume that our theorem is true for all $i \leq n$. Then, by Theorem~\ref{thm:chordalSinkEquivalent} pick a sequence of graphs $\Tour_n = X_0, X_1, \ldots, X_k = X$ such that for all $1 \leq i \leq k$ the $E(X_{i-1}) = E(X_i) \cup \{a_i \rightarrow b_i\}$ for some sink-equivalent set $\{a_i \rightarrow b_i\}$. For all $1 \leq i \leq k$, we note that
\[
\frac{\ODP(X_i,\Cycle_n)}{(1-x)^n} = \frac{\ODP(X_{i-1},\Cycle_n)}{(1-x)^n} + \frac{\ODP(X_{i-1},\Cycle_n)_{a \rightarrow b}}{x(1-x)^{n-1}} 
\]
by Theorem~\ref{thm:odp}. We now induct on $i$. The base case $i = 0$ is equivalent to
\[
\frac{\ODP(\Tour_n, \Cycle_n)}{(1-x)^n} = n \sum_{m = 0}^{\infty} m^{n-1} x^m,
\]
which is the statement of theorem \ref{thm:cycle-init}. Now, for all $1 \leq i \leq k$, we use both of our inductive hypotheses:
\begin{align*}
\frac{\ODP(X_i,\Cycle_n)}{(1-x)^n} &= \frac{\ODP(X_{i-1},\Cycle_n)}{(1-x)^n} + \frac{\ODP(X_{i-1},\Cycle_n)_{a \rightarrow b}}{x(1-x)^{n-1}}  \\
&= (-1)^n n \sum_{m = 0}^{\infty} \frac{\chi_{\conj{X_{i-1}}}(-m)}{m} x^m + \frac{\ODP(X_{i-1} \setminus \{b\},\Path_{n-1})}{(1-x)^{n-1}} \\
&= (-1)^n n \sum_{m = 0}^{\infty} \frac{\chi_{\conj{X_{i-1}}}(-m)}{m} x^m  \\
&\quad + (-1)^{n-1} (n-1) \left( \frac{nx}{n-1} \right) \sum_{m=0}^{\infty} \frac{\chi_{\conj{X_{i-1} \setminus \{b\}}}(-m)}{m} x^m \\
&= (-1)^n n \sum_{m = 0}^{\infty} \frac{1}{m}\left(\chi_{\conj{X_{i-1}}}(-m) - \chi_{\conj{X_{i-1} \setminus \{b\}}}(-m)  \right) x^m. 
\end{align*}
However, note that for any positive integer $k$, the value $\chi_{\conj{X_{i-1} \setminus \{b\}}}(k)$ is exactly the number of ways to color the graph $\conj{X_{i-1}}$ with $k$ colors such that $a$ and $b$ share the same color but no two other neighboring vertices share the same color. This is also exactly $\chi_{\conj{X_{i-1}}}(k) - \chi_{\conj{X_{i}}}(k)$. Thus:
\[
(-1)^n n\sum_{m = 0}^{\infty} \frac{1}{m}\left(\chi_{\conj{X_{i-1}}}(-m) - \chi_{\conj{X_{i-1} \setminus \{b\}}}(-m)  \right) x^m  = (-1)^n n\sum_{m = 0}^{\infty} \frac{1}{m} \chi_{\conj{X_{i}}}(-m) x^m,
\]
which proves our inductive step.
\end{proof}
We may reformulate Theorem~\ref{thm:cycle-gessel} without mentioning directed friends-and-seats graphs in a similar way to our reformulation of Theorem~\ref{thm:gessel}. For any permutation $\sigma$ of $[n]$, define the number of $G$-cyclic-descents of $\sigma$, denoted $\cdes_G(\sigma)$ to be the number of positive integers $1 \leq i \leq n$ that satisfy the following conditions:
\begin{itemize}
\item The permutation $\sigma$ satisfies that $\sigma(i) > \sigma(i+1)$, indices taken modulo $n$ with $\sigma(0) = \sigma(n)$.
\item The vertices with labels $\sigma(i)$ and $\sigma(i+1)$ are neighbors in $G$.
\end{itemize} 
The \textit{generalized cyclic Eulerian polynomial} $C_G(x)$ satisfies
\[
C_G(x) = \sum_{\sigma \in S_n} x^{\cdes_G(\sigma)}.
\]
Now, using Theorem~\ref{thm:cycle-gessel} we may write:
\begin{corollary}
Let $\chi_{G}(x)$ be the chromatic polynomial of the undirected graph $G$. We have:
\[
\frac{C_G(x)}{(1-x)^{n}} = (-1)^n n \sum_{m = 0}^{\infty} \frac{\chi_{\conj{X}}(-m)}{m} x^m.
\]
\end{corollary}

\section{Acknowledgments}
I am grateful to Tanya Khovanova, who introduced me to Eulerian numbers and mentored me throughout this project. Thanks also to Ira Gessel for consulting on this project.  I am also indebted to the MIT PRIMES-USA program for creating such a rare and amazing math research opportunity. 
\printbibliography
\end{document}